\documentclass[10pt,reqno]{amsart}
\usepackage{latexsym,amsmath,amssymb,amscd,amsfonts,amsthm, mathrsfs}
\begin{document}
\parskip=5pt

\newtheorem{theorem}{Theorem}
\newtheorem{acknowledgment}[theorem]{Acknowledgment}
\newtheorem{corollary}[theorem]{Corollary}
\newtheorem{definition}[theorem]{Definition}
\newtheorem{example}[theorem]{Example}
\newtheorem{lemma}[theorem]{Lemma}
\newtheorem{notation}[theorem]{Notation}
\newtheorem{proposition}[theorem]{Proposition}
\newtheorem{remark}[theorem]{Remark}
\newtheorem{setting}[theorem]{Setting}

\numberwithin{theorem}{section}
\numberwithin{equation}{section}

\newcommand{\1}{{\bf 1}}
\newcommand{\Ad}{{\rm Ad}}
\newcommand{\Alg}{{\rm Alg}\,}
\newcommand{\Aut}{{\rm Aut}\,}
\newcommand{\ad}{{\rm ad}}
\newcommand{\Borel}{{\rm Borel}}
\newcommand{\card}{{\rm card}\,}
\newcommand{\Ci}{{\mathscr C}^\infty}
\newcommand{\Cpol}{{\mathscr C}^\infty_{\rm pol}}
\newcommand{\Der}{{\rm Der}\,}
\newcommand{\de}{{\rm d}}
\newcommand{\ee}{{\rm e}}
\newcommand{\End}{{\rm End}\,}
\newcommand{\ev}{{\rm ev}}
\newcommand{\id}{{\rm id}}
\newcommand{\ie}{{\rm i}}
\newcommand{\GL}{{\rm GL}}
\newcommand{\gl}{{{\mathfrak g}{\mathfrak l}}}
\newcommand{\Hom}{{\rm Hom}\,}
\newcommand{\Img}{{\rm Im}\,}
\newcommand{\Ind}{{\rm Ind}}
\newcommand{\Ker}{{\rm Ker}\,}
\newcommand{\Lie}{\text{\bf L}}
\newcommand{\m}{\text{\bf m}}
\newcommand{\pr}{{\rm pr}}
\newcommand{\Ran}{{\rm Ran}\,}
\renewcommand{\Re}{{\rm Re}\,}
\newcommand{\so}{\text{so}}
\newcommand{\spa}{{\rm span}\,}
\newcommand{\Tr}{{\rm Tr}\,}
\newcommand{\Op}{{\rm Op}}
\newcommand{\U}{{\rm U}}

\newcommand{\CC}{{\mathbb C}}
\newcommand{\RR}{{\mathbb R}}
\newcommand{\TT}{{\mathbb T}}

\newcommand{\Ac}{{\mathscr A}}
\newcommand{\Bc}{{\mathscr B}}
\newcommand{\Cc}{{\mathscr C}}
\newcommand{\Dc}{{\mathscr D}}
\newcommand{\Ec}{{\mathscr E}}
\newcommand{\Fc}{{\mathscr F}}
\newcommand{\Hc}{{\mathscr H}}
\newcommand{\Jc}{{\mathscr J}}
\newcommand{\Nc}{{\mathscr N}}
\newcommand{\Oc}{{\mathscr O}}
\newcommand{\Pc}{{\mathscr P}}
\newcommand{\Rc}{{\mathscr R}}
\newcommand{\Sc}{{\mathscr S}}
\newcommand{\Tc}{{\mathscr T}}
\newcommand{\Vc}{{\mathscr V}}
\newcommand{\Uc}{{\mathscr U}}
\newcommand{\Yc}{{\mathscr Y}}
\newcommand{\Wig}{{\mathscr W}}

\newcommand{\Bg}{{\mathfrak B}}
\newcommand{\Fg}{{\mathfrak F}}
\newcommand{\Gg}{{\mathfrak G}}
\newcommand{\Ig}{{\mathfrak I}}
\newcommand{\Jg}{{\mathfrak J}}
\newcommand{\Lg}{{\mathfrak L}}
\newcommand{\Pg}{{\mathfrak P}}
\newcommand{\Sg}{{\mathfrak S}}
\newcommand{\Xg}{{\mathfrak X}}
\newcommand{\Yg}{{\mathfrak Y}}
\newcommand{\Zg}{{\mathfrak Z}}

\newcommand{\ag}{{\mathfrak a}}
\newcommand{\bg}{{\mathfrak b}}
\newcommand{\dg}{{\mathfrak d}}
\renewcommand{\gg}{{\mathfrak g}}
\newcommand{\hg}{{\mathfrak h}}
\newcommand{\kg}{{\mathfrak k}}
\newcommand{\mg}{{\mathfrak m}}
\newcommand{\n}{{\mathfrak n}}
\newcommand{\og}{{\mathfrak o}}
\newcommand{\pg}{{\mathfrak p}}
\newcommand{\sg}{{\mathfrak s}}
\newcommand{\tg}{{\mathfrak t}}
\newcommand{\ug}{{\mathfrak u}}
\newcommand{\zg}{{\mathfrak z}}

\newcommand{\ZZ}{\mathbb Z}
\newcommand{\NN}{\mathbb N}
\newcommand{\BB}{\mathbb B}

\newcommand{\ep}{\varepsilon}

\newcommand{\hake}[1]{\langle #1 \rangle }

\newcommand{\scalar}[2]{\langle #1 ,#2 \rangle }
\newcommand{\vect}[2]{(#1_1 ,\ldots ,#1_{#2})}
\newcommand{\norm}[1]{\Vert #1 \Vert }
\newcommand{\normrum}[2]{{\norm {#1}}_{#2}}

\newcommand{\upp}[1]{^{(#1)}}
\newcommand{\p}{\partial}

\newcommand{\opn}{\operatorname}
\newcommand{\slim}{\operatornamewithlimits{s-lim\,}}
\newcommand{\sgn}{\operatorname{sgn}}

\newcommand{\seq}[2]{#1_1 ,\dots ,#1_{#2} }
\newcommand{\loc}{_{\opn{loc}}}

\title{A survey on Weyl calculus for representations of nilpotent Lie groups}

\author{Ingrid Belti\c t\u a and Daniel Belti\c t\u a}
\address{Institute of Mathematics ``Simion Stoilow'' 
of the Romanian Academy, 
P.O. Box 1-764, Bucharest, Romania}
\email{Ingrid.Beltita@imar.ro}
\email{Daniel.Beltita@imar.ro}

\begin{abstract}
 We survey some aspects of the pseudo-differential Weyl calculus 
for irreducible unitary representations of nilpotent Lie groups, 
ranging from the classical ideas to recently obtained results.  
The classical Weyl-H\"ormander calculus is recovered for 
the Schr\"odinger representation of the Heisenberg group. 
Our discussion concerns various extensions of this classical situation 
to arbitrary nilpotent Lie groups and to some infinite-dimensional Lie groups 
that allow us to handle the magnetic pseudo-differential calculus. 
\end{abstract}

\maketitle


\section{Introduction}\label{Sect1}

The Weyl calculus was first constructed in \cite{We28} for 
the purposes of quantum mechanics.  
It was afterwards investigated 
and extended to an abstract setting in \cite{An69} and \cite{An72}.  
This calculus was also taken up in \cite{Hor79} and made 
into a pseudo-differential calculus which plays a central role 
in the theory of partial differential equations and in many of 
its applications to mathematical physics. 

In the present paper we provide a brief discussion of the Weyl calculus 
and its later extensions in the framework provided by 
nilpotent Lie groups and their representation theory, ranging from 
some classical ideas of \cite{We28} and \cite{Hor79}, 
going through the important developments in \cite{Pe94}, 
and concluding by recent results from \cite{BB09a}, \cite{BB09b}, and \cite{BB09c}. 
We shall emphasize the role of the coadjoint orbits 
and the corresponding unitary representations 
as a natural background for the Weyl calculus. 
For instance the classical phase space ${\RR}^n\times{\RR}^n$ 
should be thought of as a coadjoint orbit of the Heisenberg group, 
corresponding to the Schr\"odinger representation, 
which leads to a transparent description of 
the classical Weyl calculus on ${\RR}^n$ 
as a very special case of the calculus constructed in \cite{Pe94} 
for arbitrary nilpotent Lie groups
(see Sections \ref{Sect2}~and~\ref{Sect3} 
below). 
The same idea allowed us to show in \cite{BB09a} that 
the magnetic pseudo-differential calculus of \cite{MP04} 
is governed by a certain infinite-dimensional Lie group, 
in the sense that it can be constructed as a Weyl quantization of 
a certain finite-dimensional coadjoint orbit of that group 
and the symbol spaces for the magnetic calculus are actually 
function spaces on that orbit
(see Section~\ref{Sect4}). 

For the sake of simplicity let us mention here that there 
exist other interesting lines of investigation on symbol calculi
for nilpotent Lie groups, which have however a different flavour and 
therefore we do not discuss them here; 
see for instance the papers \cite{Me83}, \cite{How84} and \cite{Gl07}. 
See also \cite{Mi82}, \cite{Mi86}, \cite{HRW84}, and \cite{Ma91} 
for examples and remarks on the relationship 
between a symbol calculus and a coadjoint orbit. 

\textbf{Notation.}
Throughout the paper we denote by $\Sc(\Vc)$ the Schwartz space 
on a finite-dimensional real vector space~$\Vc$. 
That is, $\Sc(\Vc)$ is the set of all smooth functions 
that decay faster than any polynomial together with 
their partial derivatives of arbitrary order. 
Its topological dual ---the space of tempered distributions on $\Vc$--- 
is denoted by $\Sc'(\Vc)$. 
We use the notation $\Cpol(\Vc)$ for the space 
of smooth functions that grow polynomially together with 
their partial derivatives of arbitrary order. 
In Section~\ref{Sect2} 
we shall also have the occasion to 
use these notions with $\Vc$ replaced by a coadjoint orbit 
of a nilpotent Lie group. 
In this situation we need the notion of polynomial structure 
on a manifold; see Sect.~1 in \cite{Pe89} for details. 
We use $\langle\cdot,\cdot\rangle$ to denote any duality pairing between 
finite-dimensional real vector spaces whose meaning is clear 
from the context. 
The Lebesgue measures on linear spaces (or Liouville measures 
on coadjoint orbits) and on their duals 
will always be normalized such that 
the corresponding Fourier transforms are unitary operators 
on the $L^2$-spaces.

\section{The classical pseudo-differential Weyl calculus on ${\RR}^n$}\label{Sect2}

\subsection{The earlier perspective on the Weyl calculus}
The Weyl calculus was introduced in \cite{We28} as 
a quantization procedure, that is, 
a natural correspondence between 
the classical observables and the quantum ones. 
More specifically, let $P_1,\dots,P_n$ be 
the quantum \emph{momentum operators} and $Q_1,\dots,Q_n$ 
the quantum \emph{position operators} in $L^2({\RR}^n)$. 
Recall that for $f\in\Sc({\RR}^n)$  and 
a generic point $q=(q_1,\dots,q_n)\in{\RR}^n$ we have 
for $j=1,\dots,n$, 
$$(Q_jf)(q)=q_jf(q) \text{ and }
P_jf=\frac{1}{\ie}\frac{\partial f}{\partial q_j}.$$
One of the remarkable properties of these operators is 
that for arbitrary points $p=(p_1,\dots,p_n)\in{\RR}^n$ and 
$q=(q_1,\dots,q_n)\in{\RR}^n$ the linear combination 
$$p\cdot Q+q\cdot P:=p_1Q_1+\cdots+p_nQ_n+q_1P_1+\cdots+q_nP_n $$
defines a self-adjoint operator in $L^2({\RR}^n)$, 
which in turn gives rise to a unitary operator $\exp(\ie(p\cdot Q+q\cdot P))$ 
that leaves $\Sc({\RR}^n)$ invariant. 

Now we can use these remarks to make the following definition 
of the \emph{pseudo-differential Weyl-H\"ormander calculus} on ${\RR}^n$ 
(see~\cite{Hor79}). 

\begin{definition}\label{wh}
\normalfont
For every $a\in\Sc({\RR}^n\times{\RR}^n)$ 
we define the corresponding \emph{pseudo-differential operator} 
$a(Q,P)$ 
by 
\begin{equation}\label{wh_eq}
a(Q,P)f=\iint\limits_{{\RR}^n\times{\RR}^n}
\widehat{a}(p,q)\exp(\ie(p\cdot Q+q\cdot P))f
\,\de p\de q 
\end{equation}
for arbitrary $f\in\Sc({\RR}^n)$, where $\widehat{a}\in\Sc({\RR}^n\times{\RR}^n)$ 
stands for the Fourier transform of the \emph{symbol}~$a$. 
\qed
\end{definition}

\begin{remark}\label{wh_dual}
\normalfont
For every tempered distribution $a\in\Sc'({\RR}^n\times{\RR}^n)$ 
we can interpret the integral~\eqref{wh_eq} in the distributional sense 
and thus define the corresponding pseudo-differential operator 
$a(Q,P)\colon\Sc({\RR}^n)\to\Sc'({\RR}^n)$. 
One can also compute the distribution kernel of 
the latter operator. More specifically, 
we have  
\begin{equation}\label{hoermander}
(a(Q,P)f)(q)
=\iint\limits_{{\RR}^n\times{\RR}^n}
a\bigl(\frac{q+q'}{2},p\bigr)e^{i(q-q')\cdot p}f(q')\de p\, \de q'
\end{equation}
for $f\in\Sc({\RR}^n)$ 
(see \cite{Hor07} for more details).
\qed 
\end{remark}

\subsection{Weyl calculus from the perspective of the Heisenberg group} 
A particularly deep insight into the Weyl calculus 
of Definition~\ref{wh} comes from taking into account 
the commutation relations satisfied by the operators 
$Q_1,\dots,Q_n$ and $P_1,\dots,P_n$, namely 
$$[Q_j,Q_k]=[P_j,P_k]=0\text{ and }[Q_j,P_k]=\delta_{jk}\ie \cdot I, $$
where $\delta_{jk}$ stands for the Kronecker's delta 
and $I$ denotes the identity operator on $L^2({\RR}^n)$. 
It follows from these commutation relations that 
we actually have to deal with the Schr\"odinger representation 
of the Heisenberg group from the following definition. 

\begin{definition}\label{hhs}
\normalfont
For every integer $n\ge1$ let `$\cdot$' denote 
the Euclidean scalar product on ${\RR}^n$. 
We introduce the \emph{Heisenberg algebra} 
$\hg_{2n+1}={\RR}^n\times{\RR}^n\times{\RR}$ 
with the bracket 
$$[(q,p,t),(q',p',t')]=[(0,0,p\cdot q'-p'\cdot q)]. $$
The \emph{Heisenberg group} 
${\mathbb H}_{2n+1}$ is just $\hg_{2n+1}$ thought of as a group 
with the multiplication~$\ast$ 
defined by 
$$X\ast Y=X+Y+\frac{1}{2}[X,Y]. $$
The unit element is $0\in{\mathbb H}_{2n+1}$ 
and the inversion mapping given by $X^{-1}:=-X$.
\qed
\end{definition}

See also \cite{How80}, \cite{Fo89}, \cite{FG92} and Ch.~9 in \cite{Gr01} 
for discussions on the Heisenberg group and its importance 
for the harmonic analysis on ${\RR}^n$. 

To define the Schr\"odinger representation we first exhibit 
the Heisenberg group as a semidirect product and then use a natural representation 
of that semidirect product. 
The details are recorded in the following remark. 

\begin{remark}\label{sch}
\normalfont
Consider ${\RR}^{n+1}\simeq{\RR}^n\times{\RR}$ and the natural representation 
$$\rho\colon({\RR}^n,+)\to\End({\RR}^{n+1}),\quad 
\rho(q)(p,t)=(p, p\cdot q+t).$$ 
Then it is straightforward to check that the mapping 
$$\Psi\colon{\mathbb H}_{2n+1}\to{\RR}^{n+1}\rtimes_\rho{\RR}^n,\quad 
(q,p,t)\mapsto\bigl((p,\textstyle{\frac{1}{2}}p\cdot q+t),q\bigr)$$
is an isomorphism of Lie groups. 

Let $\Pc_1({\RR}^n)$ be the linear space of real polynomial functions 
of degree $\le1$ on ${\RR}^n$ and 
note that $\Pc_1({\RR}^n)$ is linearly isomorphic to ${\RR}^n\times{\RR}$, 
since for every 
$\varphi\in\Pc_1({\RR}^n)$ there exist uniquely determined $\xi\in{\RR}^n$ and $t\in{\RR}$ 
such that $\varphi(x)=\xi\cdot x+t$ for every $x\in{\RR}^n$. 
We also get a linear representation 
$$\rho\colon({\RR}^n,+)\to\End(\Pc_1({\RR}^n)),\quad 
\rho(q)\varphi=\varphi(q+\cdot)$$
and then ${\mathbb H}_{2n+1}\simeq\Pc_1({\RR}^n)\rtimes_\rho{\RR}^n$ in view of the above paragraph 
(see also Ex.~2.6 in~\cite{BB09a}). 

Note that the semidirect product $\Pc_1({\RR}^n)\rtimes_\rho{\RR}^n$ has a natural unitary 
representation on $L^2({\RR}^n)$ given for arbitrary $(\varphi,q)\in\Pc_1({\RR}^n)\rtimes_\rho{\RR}^n$ by  
$$(\forall f\in L^2({\RR}^n))\quad \pi(\varphi,q)f=\ee^{\ie\varphi(\cdot)}f(q+\cdot).$$ 
Then the above explicit isomorphisms 
${\mathbb H}_{2n+1}\simeq{\RR}^{n+1}\rtimes_\rho{\RR}^n\simeq\Pc_1({\RR}^n)\rtimes_\rho{\RR}^n$, 
lead to the unitary representation 
$\pi\colon{\mathbb H}_{2n+1}\to\Bc(L^2({\RR}^n))$ 
defined by  
\begin{equation}\label{sch_eq}
(\pi(q,p,t)f)(x)=\ee^{\ie(p\cdot x+\frac{1}{2}p\cdot q+t)}f(q+x) 
\text{ for a.e. }x\in{\RR}^n
\end{equation}
for arbitrary $f\in L^2({\RR}^n)$ and $(q,p,t)\in{\mathbb H}_{2n+1}$. 
This is the \emph{Schr\"odinger representation} of the Heisenberg group ${\mathbb H}_{2n+1}$.  
\qed
\end{remark}

The Schr\"odinger representation of the Heisenberg group
provides the natural background for 
the pseudo-differential Weyl calculus on ${\RR}^n$. 
To illustrate this idea, we shall describe a  
condition on the symbol $a\in\Sc'({\RR}^n\times{\RR}^n)$ 
which ensures that the pseudo-differential operator 
$\Op(a)$ is bounded on $L^2({\RR}^n)$. 
This $L^2$-boundedness theorem was obtained in \cite{GH99} and is stated 
in terms of the modulation spaces introduced below. 

\begin{definition}\label{mod_def}
\normalfont 
Let us denote by $\hake{\cdot,\cdot}\colon\Sc'({\RR}^n)\times\Sc({\RR}^n)\to\CC$ 
the usual duality pairing. 
Assume that $1\le r,s\le\infty$ and $\phi\in\Sc({\RR}^n)$ and define 
for every tempered distribution $b\in\Sc'({\RR}^n)$ the corresponding 
\emph{ambiguity function} 
$$\Ac_\phi b\colon{\RR}^n\times{\RR}^n\to\CC,\quad 
(\Ac_\phi b)(q,p)=\hake{b,\overline{\pi(q,p,0)\phi}} $$
and then 
$$\Vert b\Vert_{M^{r,s}_\phi}
=\Bigl(\int\limits_{{\RR}^n}\Bigl(\int\limits_{{\RR}^n}
\vert(\Ac_\phi b)(q,p)\vert^s\de q \Bigr)^{r/s}
\de p\Bigr)^{1/r}\in[0,\infty] $$
with the usual conventions if $r$ or $s$ is infinite. 
Then the space 
$$M^{r,s}({\RR}^n):=\{b\in\Sc'({\RR}^n)\mid\Vert b\Vert_{M^{r,s}_\phi}<\infty\}$$ 
does not depend on $\phi\in\Sc({\RR}^n)$ and 
is called a \emph{modulation space} on ${\RR}^n$. 
\qed
\end{definition}

\begin{theorem}\label{gh} 
For every $a\in M^{\infty,1}({\RR}^{2n})$ the corresponding pseudo-differential operator 
$\Op(a)$ is bounded on $L^2({\RR}^n)$. 
Moreover, for every $\phi\in\Sc({\RR}^{2n})$ there exists a constant $C_\phi>0$ 
such that if $a\in M^{\infty,1}({\RR}^{2n})$, then 
$\Vert\Op(a)\Vert\le C_\phi\Vert a\Vert_{M^{1,\infty}_\phi}$. 
\end{theorem}

\begin{proof}
See Th.~1.1 in \cite{GH99}.  
\end{proof}

\section{Weyl calculus for irreducible unitary representations}\label{Sect3}

We shall briefly describe some of the remarkable results 
obtained in \cite{Pe94} (relying on \cite{Pe84}, \cite{Pe88}, and \cite{Pe89})) on the Weyl calculus for irreducible unitary representations 
of nilpotent Lie groups.

\subsection{Preduals for coadjoint orbits}

\begin{setting}\label{predual_sett}
\normalfont 
We shall use the following notation:
\begin{enumerate}
 \item Let $G$ be a connected, simply connected, nilpotent Lie group with the Lie algebra~$\gg$.  
 Then the exponential map $\exp_G\colon\gg\to G$ is a diffeomorphism 
 with the inverse denoted by $\log_G\colon G\to\gg$. 
 \item We denote by $\gg^*$ the linear dual space to $\gg$ and 
  by $\hake{\cdot,\cdot}\colon\gg^*\times\gg\to{\RR}$ the natural duality pairing. 
 \item Let $\xi_0\in\gg^*$ with the corresponding coadjoint orbit $\Oc:=\Ad_G^*(G)\xi_0\subseteq\gg^*$.  
 \item The \emph{isotropy group} at $\xi_0$ is $G_{\xi_0}:=\{g\in G\mid\Ad_G^*(g)\xi_0=\xi_0\}$ 
 with the corresponding \emph{isotropy Lie algebra} $\gg_{\xi_0}=\{X\in\gg\mid\xi_0\circ\ad_{\gg}X=0\}$. 
 If we denote the \emph{center} of $\gg$ by $\zg:=\{X\in\gg\mid[X,\gg]=\{0\}\}$, 
 then it is clear that $\zg\subseteq\gg_{\xi_0}$. 
 \item Let $n:=\dim\gg$ and fix a sequence of ideals in $\gg$, 
$$\{0\}=\gg_0\subset\gg_1\subset\cdots\subset\gg_n=\gg$$
such that $\dim(\gg_j/\gg_{j-1})=1$ and $[\gg,\gg_j]\subseteq\gg_{j-1}$ 
for $j=1,\dots,n$. 
 \item Pick any $X_j\in\gg_j\setminus\gg_{j-1}$ for $j=1,\dots,n$, 
so that the set $\{X_1,\dots,X_n\}$ will be a \emph{Jordan-H\"older basis} in~$\gg$. 
\end{enumerate}
\qed 
\end{setting}

\begin{definition}\label{jump_def}
\normalfont
 Consider the set of \emph{jump indices} of the coadjoint orbit $\Oc$ 
with respect to the aforementioned Jordan-H\"older basis $\{X_1,\dots,X_n\}\subset\gg$, 
$$e:=\{j\in\{1,\dots,n\}\mid \gg_j\not\subseteq\gg_{j-1}+\gg_{\xi_0}\}
=\{j\in\{1,\dots,n\}\mid X_j\not\in\gg_{j-1}+\gg_{\xi_0}\}$$ 
and then define the corresponding \emph{predual of the coadjoint orbit}~$\Oc$, 
$$\gg_e:=\spa\{X_j\mid j\in J_{\xi_0}\}\subseteq\gg.$$
We note the direct sum decomposition $\gg=\gg_{\xi_0}\dotplus\gg_e$. 
\qed 
\end{definition}

\begin{remark}\label{jump_remark}
\normalfont
Let $\{\xi_1,\dots,\xi_n\}\subset\gg^*$ be the dual basis for $\{X_1,\dots,X_n\}\subset\gg$. 
Then the coadjoint orbit $\Oc$ can be described in terms of the jump indices 
mentioned in Definition~\ref{jump_def}. 
More specifically, if we denote 
$$\gg^*_{\Oc}:=\spa\{\xi_j\mid j\in e\} 
\quad\text{ and }\quad\gg^\perp_{\Oc}:=\spa\{\xi_j\mid j\not\in e\},$$ 
then the coadjoint orbit  
$\Oc\subset\gg^*\simeq\gg_e^*\times\gg_e^\perp$ 
is the graph of a certain \emph{polynomial mapping} $\gg_e^*\to\gg_e^\perp$. 
This leads to the following pieces of information on $\Oc$: 
\begin{enumerate}
 \item $\dim\Oc=\dim\gg_e=\card e=:d$;  
 \item if we let $j_1<\cdots<j_d$ such that $e=\{j_1,\dots,j_d\}$, 
then the mapping 
$$\Oc\to{\RR}^d,\quad \xi\to(\hake{\xi,X_{j_1}},\dots,\hake{\xi,X_{j_d}}) $$
is a global chart which takes the Liouville measure of $\Oc$ 
to a Lebesgue measure on~${\RR}^d$. 
\end{enumerate}
See for instance Lemma~1.6.1 in \cite{Pe89} for more details and proofs 
for these assertions. 
\qed
\end{remark}

\subsection{Weyl calculus for unitary irreducible representations}
\begin{setting}\label{calculus_sett}
\normalfont
In addition to Setting~\ref{predual_sett} 
we now fix  some further notation: 
\begin{enumerate}
 \item Let $\pi\colon G\to\Bc(\Hc)$ be any unitary irreducible representations 
associated with the coadjoint orbit $\Oc$ by Kirillov's theorem (\cite{Ki62}). 
\item We define the Fourier transform $\Sc(\Oc)\to\Sc(\gg_e)$ by 
$$(\forall X\in\gg_e)\quad \widehat{a}(X)=\int\limits_{\Oc}\ee^{-\ie\hake{\xi,X}}a(\xi)\de\xi  $$
for every $a\in\Sc(\Oc)$, where $\de\xi$ stands for a Liouville measure on $\Oc$ 
(see also Remark~\ref{jump_remark}). 
This Fourier transform is invertible (Lemma 4.1.1 in \cite{Pe94}) 
and its inverse is denoted by 
$\Sc(\gg_e)\to\Sc(\Oc)$, $a\mapsto \check a$. 
\end{enumerate}
\qed 
\end{setting}

\begin{remark}\label{smoothness}
\normalfont
Let us consider the space of \emph{smooth vectors} for the representation~$\pi$, 
$$\Hc_\infty:=\{v\in\Hc\mid \pi(\cdot)v\in\Ci(G,\Hc)\}. $$
Then $\Hc_\infty$ has a natural structure of Fr\'echet space 
which carries the \emph{derivate representation} $\de\pi\colon\gg\to\End(\Hc_\infty)$. 
The latter map is a homomorphism of Lie algebras defined by 
$$(\forall X\in\gg,v\in\Hc_\infty)\quad\de\pi(X)v=\frac{\de}{\de t}\Big{\vert}_{t=0}\pi(\exp_G(tX))v. $$
Now let us denote by $\Sg_p(\Hc)$ the Schatten ideals of operators on $\Hc$ for $1\le p\le\infty$. 
Consider the unitary representation 
$\Pi\colon G\times G\to\Bc(\Sg_2(\Hc))$ defined by 
$$(\forall g_1,g_2\in G)(\forall T\in\Sg_2(\Hc))\quad 
\Pi(g_1,g_2)T=\pi(g_1)T\pi(g_2)^{-1}.$$
It is not difficult to see that $\Pi$ is strongly continuous.  
The corresponding space of smooth vectors is denoted by $\Bc(\Hc)_\infty$ 
and is called the space of \emph{smooth operators} for the representation~$\pi$. 
One can prove that actually $\Bc(\Hc)_\infty\subseteq\Sg_1(\Hc)$. 

For an alternative description of $\Bc(\Hc)_\infty$ 
let $\gg_{\CC}:=\gg\otimes_{{\RR}}\CC$ be the complexification of $\gg$ 
with the corresponding universal associative enveloping algebra $\U(\gg_{\CC})$. 
Then the aforementioned homomorphism of Lie algebras $\de\pi$ has a unique extension 
to a homomorphism of unital associative algebras $\de\pi\colon\U(\gg_{\CC})\to\End(\Hc_\infty)$. 
One can prove that for $T\in\Bc(\Hc)$ we have $T\in\Bc(\Hc)_\infty$ if and only if  $T(\Hc)+T^*(\Hc)\subseteq\Hc_\infty$ and $\de\pi(u)T,\de\pi(u)T^*\in\Bc(\Hc)$ 
for every $u\in\U(\gg_{\CC})$. 
(See subsect.~1.2 in \cite{Pe94}.)
\qed
\end{remark}

\begin{definition}\label{calc_def}
\normalfont 
The \emph{Weyl calculus} $\Op^\pi(\cdot)$ for the unitary representation~$\pi$ is defined 
for every $a\in\Sc(\Oc)$ by 
$$\Op^\pi(a)=\int\limits_{\gg_e}\widehat{a}(X)\pi(\exp_GX)\de X\in\Bc(\Hc). $$
We call $\Op^\pi(a)$ is the \emph{pseudo-differential operator} with 
the \emph{symbol} $a\in\Sc(\Oc)$. 
\qed 
\end{definition}

\begin{theorem}\label{pedersen}
The Weyl calculus has the following properties: 
\begin{enumerate}
\item\label{pedersen_item1} For every symbol $a\in\Sc(\Oc)$ we have $\Op^\pi(a)\in\Bc(\Hc)_\infty$ 
and the mapping $\Sc(\Oc)\to\Bc(\Hc)_\infty$, $a\mapsto\Op^\pi(a)$ is 
a linear topological isomorphism. 
\item\label{pedersen_item1bis} 
For every $T\in\Bc(\Hc)_\infty$ we have $T=\Op^\pi(a)$, 
where $a\in\Sc(\Oc)$ satisfies the condition $\widehat{a}(X)=\Tr(\pi(\exp_G X)^{-1}A)$ for every $X\in\gg_e$. 
\item\label{pedersen_item2} For every $a,b\in\Sc(\Oc)$ we have 
\begin{enumerate}
\item\label{pedersen_item2a} $\Op^\pi(\bar{a})=\Op^\pi(a)^*$; 
\item\label{pedersen_item2b} $\Tr(\Op^\pi(a))=\int\limits_{\Oc}a(\xi)\de\xi$; 
\item\label{pedersen_item2c} $\Tr(\Op^\pi(a)\Op^\pi(b))=\int\limits_{\Oc}a(\xi)b(\xi)\de\xi$;
\item\label{pedersen_item2d} $\Tr(\Op^\pi(a)\Op^\pi(b)^*)=\int\limits_{\Oc}a(\xi)\overline{b(\xi)}\de\xi$. 
\end{enumerate}
\end{enumerate}
\end{theorem}

\begin{proof}
See Th.~4.1.4 and Th.~2.2.7 in \cite{Pe94}. 
\end{proof}

\begin{remark}\label{calc_dual}
\normalfont 
Let $\Bc(\Hc)_\infty^*$ be the topological dual of the Fr\'echet space $\Bc(\Hc)_\infty$ 
and denote by $\hake{\cdot,\cdot}$ either of the duality pairings 
$$\Bc(\Hc)_\infty^*\times\Bc(\Hc)_\infty\to\CC 
\text{ and }\Sc'(\Oc)\times\Sc(\Oc)\to\CC.$$ 
Then for every tempered distribution $a\in\Sc'(\Oc)$ we can use Theorem~\ref{pedersen}\eqref{pedersen_item1} 
to define 
$\Op^\pi(a)\in\Bc(\Hc)_\infty^*$ such that  
$$(\forall b\in\Sc(\Oc))\quad \hake{\Op^\pi(a),\Op^\pi(b)}=\hake{a,b} $$
Just as in Definition~\ref{calc_def} we call $\Op^\pi(a)$ the \emph{pseudo-differential operator} with 
the \emph{symbol} $a\in\Sc'(\Oc)$. 
Note that if actually $a\in\Sc(\Oc)$, 
then the present notation agrees with Definition~\ref{calc_def} 
because of Theorem~\ref{pedersen}\eqref{pedersen_item2c}. 

The continuity properties of the above pseudo-differential operators 
can be investigated by using modulation spaces of symbols; see \cite{BB09c} for details. 
Specifically, one can extend Definition~\ref{mod_def} in order 
to introduce modulation spaces $M^{r,s}_\phi(\pi)$ for every unitary irreducible representation 
$\pi\colon G\to\Bc(\Hc)$.  
We always have $\Hc_\infty\subseteq M^{r,s}_\phi(\pi)$ 
and $M^{2,2}_\phi(\pi)=\Hc$. 
If the representation $\pi$ is square-integrable modulo the center of $G$ 
and $\Oc$ is the corresponding coadjoint orbit, 
then there exists a natural representation $\pi^{\#}\colon G\ltimes G\to\Bc(L^2(\Oc))$ 
such that the Weyl calculus $\Op^\pi(\cdot)$ defines a continuous linear mapping 
from the modulation space $M^{\infty,1}(\pi^{\#})$ into the space of \emph{bounded} linear operators 
on~$\Hc$. 
Theorem~\ref{gh} is recovered in the special case when $\pi$ is the Schr\"odinger representation 
of the Heisenberg group ${\mathbb H}_{2n+1}$. 
\qed 
\end{remark}

In the following statement we shall use the notation
$$(\forall X\in\gg)\quad \psi^X\colon\Oc\to\CC,\quad \psi^X(\xi)=\ie\hake{\xi,X}. $$
For every integer $m\ge0$ and every $X\in\gg$ 
the $m$-th power $(\psi^X)^m$ can be thought of as an element in $\Sc'(\Oc)$ 
in the usual way. 

\begin{theorem}\label{pedersen_dual} 
The Weyl calculus with symbols in $\Sc'(\Oc)$ has the following properties: 
\begin{enumerate}
\item\label{pedersen_dual_item1} 
The mapping $\Sc'(\Oc)\to\Bc(\Hc)_\infty^*$, $a\mapsto\Op^\pi(a)$ is 
a linear topological isomorphism. 
\item\label{pedersen_dual_item2} 
For every $X\in\gg$ we have 
$\Op^\pi(\psi^X)=\de\pi(X)$. 
\item\label{pedersen_dual_item3} 
If $Y\in\gg_e$, then for every integer $m\ge0$ we have $\Op^\pi((\psi^Y)^m)=\de\pi(Y)^m$.
\end{enumerate}
\end{theorem}

\begin{proof}
See Th.~4.1.4(7)--(8) in \cite{Pe94}. 
\end{proof}

\begin{definition}\label{moyal}
\normalfont 
Note that $\Bc(\Hc)_\infty$ is an involutive associative subalgebra of $\Bc(\Hc)$ 
as an easy consequence of the alternative description in Remark~\ref{smoothness}. 
It then follows by Theorem~\ref{pedersen}\eqref{pedersen_item1} 
that there exists an uniquely defined bilinear associative \emph{Moyal product} 
$$\Sc(\Oc)\times\Sc(\Oc)\to\Sc(\Oc),\quad (a,b)\mapsto a\#^\pi b $$
such that 
$$(\forall a,b\in\Sc(\Oc))\quad \Op^\pi(a\#^\pi b)=\Op^\pi(a)\Op^\pi(b). $$
Thus $\Sc(\Oc)$ is made into an involutive associative algebra 
such that the mapping $\Sc(\Oc)\to\Bc(\Hc)_\infty$, $a\mapsto\Op^\pi(a)$ 
is an algebra isomorphism. 
\qed
\end{definition}

\begin{remark}\label{maillard}
\normalfont
It was proved in \cite{Ma07} that if $\pi$ is a square-integrable representation, 
then the Moyal product $\#^\pi$ is a star product and an explicit formal expansion 
was obtained. 
\qed
\end{remark}

\begin{remark}\label{remark3.12}
\normalfont
In the case of the Heisenberg group (see Definition~\ref{hhs}), 
let us perform the identification $\hg_{2n+1}^*\simeq\RR^n\times\RR^n\times\RR$ 
by means of the Euclidean structure of $\RR^{2n+1}$.  
Then the only nontrivial coadjoint orbits are of the form 
$\Oc=\RR^n\times\RR^n\times\{t\}$ 
with $t\in\RR\setminus\{0\}$, and $\Sc(\Oc)$ can be naturally identified 
with $\Sc(\RR^n\times\RR^n)$. 
When $t=1$,  the corresponding irreducible representation of ${\mathbb H}_{2n+1}$ 
is given by \eqref{sch_eq}, while the predual to the coadjoint orbit 
can be identified to $\RR^n\times\RR^n$. 
Then it is easy to see that 
$$a(Q,P)=\Op^\pi(a) \quad\text{for}\quad a\in\Sc(\RR^n\times\RR^n).$$
Thus the classical Weyl pseudo-differential calculus on $\RR^n$ 
can be directly obtained as a special calculus 
of the Weyl calculus of \cite{Pe94} described in Definition~\ref{calc_def} above. 
\qed
\end{remark}

\section{Magnetic Weyl calculus on nilpotent Lie groups}\label{Sect4}

The magnetic pseudo-differential Weyl calculus on ${\RR}^n$ 
developed in \cite{MP04}, \cite{IMP07} and other works was 
motivated by problems in quantum mechanics.  
In the present section we describe results from \cite{BB09a} 
(see also \cite{BB09b}) which show that these constructions 
can be extended to any simply connected nilpotent Lie group and 
can be related to the program of Weyl quantization for 
coadjoint orbits (see e.g., \cite{Wi89} and \cite{Ca07}) 
and to the Weyl calculus for irreducible unitary representations 
discussed in Section~\ref{Sect3}. 

A magnetic potential on a Lie group~$G$ is simply a 1-form $A\in\Omega^1(G)$, 
and the corresponding magnetic field is $B=dA\in\Omega^2(G)$. 
The purpose of a magnetic pseudo-differential calculus on $G$ is to facilitate the investigation 
on first-order linear differential operators of the form 
$-\ie P_0+A(Q)P_0$, 
where 
$P_0$ is a right invariant vector field on $G$ and $A(Q)P_0$ stands 
for the operator defined by the multiplication by the function 
obtained by applying the (non-invariant) 1-form $A$ 
to the vector field $P_0$ at every point in~$G$. 
In the special case of the abelian group $G=({\RR}^n,+)$ 
we have $A=A_1\de x_1+\dots+A_n\de x_n\in\Omega^1({\RR}^n)$ 
and the above operators 
are precisely the linear partial differential operators 
determined by vectors $P_0=(p_1,\dots,p_n)\in{\RR}^n$, 
\begin{equation*}
\ie\Bigl(p_1\frac{\partial}{\partial x_1}+\cdots+ p_n\frac{\partial}{\partial x_n}\Bigr)
+\Bigl(p_1A_1(Q)+\cdots+p_nA_n(Q)\Bigr)
=\sum_{j=1}^np_j\Bigl(\ie \frac{\partial}{\partial x_j}+A_j(Q)\Bigr)
\end{equation*}
where we denote by $A_1(Q),\dots,A_n(Q)$ the operators of multiplication 
by the coefficients of the 1-form~$A$. 
In the non-magnetic case (i.e., $dA=B=0$), 
we get precisely the operators involved in the classical Weyl calculus 
on ${\RR}^n$ discussed in Section~\ref{Sect2}.

\subsection{Magnetic Weyl calculus}

\begin{setting}\label{o1}
\normalfont 
Let us summarize the framework for the present section. 
\begin{enumerate}
\item A connected, simply connected, nilpotent Lie group $G$ is 
identified to its Lie algebra $\gg$ by means of the exponential map. 
We denote by $\ast$ the Baker-Campbell-Hausdorff multiplication 
on $\gg$, so that $G=(\gg,\ast)$. 
\item 
The cotangent bundle $T^*G$ is a trivial bundle and 
we perform the identification 
\begin{equation*}
T^*G\simeq\gg\times\gg^*
\end{equation*}
by using the trivialization by left translations. 
\item $\Fc$ is an admissible function space on the Lie group $G$ 
(see  Def.~2.8 in \cite{BB09a});  
in particular, $\Fc$ is invariant under translations to the left on $G$ 
and is endowed with a locally convex topology 
such that we have continuous inclusion mappings  $\gg^*\hookrightarrow \Fc\hookrightarrow\Ci(G)$. 
For instance $\Fc$ can be the whole space $\Ci(G)$ or the space $\Cpol(G)$ 
of smooth functions with polynomial growth. 
See however \cite{BB09b}  for specific situations when $\dim\Fc<\infty$.
\item The semidirect product $M=\Fc\rtimes_\lambda G$ is an infinite-dimensional Lie group 
in general, whose Lie algebra is $\mg=\Fc\rtimes_{\dot\lambda}\gg$. 
We refer to \cite{Ne06} for basic facts on infinite-dimensional Lie groups. 
\item We endow $\gg$ and its dual space $\gg^*$ with Lebesgue measures 
suitably normalized such that the Fourier transform $L^2(\gg)\to L^2(\gg^*)$ 
is a unitary operator, and we denote $\Hc=L^2(\gg)$. 
\item\label{o1_item6} 
We define a unitary representation $\pi\colon M\to\Bc(\Hc)$ by 
$$(\pi(\phi,X)f)(Y)=\ee^{\ie\phi(Y)}f((-X)\ast Y)$$
for $(\phi,X)\in M$, $f\in\Hc$, and $Y\in\gg$. 
\item The \emph{magnetic potential} $A\in\Omega^1(G)$ is 
a smooth differential 1-form whose coefficients belong to~$\Fc$. 
That is, a smooth mapping $A\colon\gg\to\gg^*$, $X\mapsto A_X$, with polynomial growth 
such that for every $X\in\gg$ we have $\langle A_\bullet,(R_\bullet)'_0X\rangle\in\Fc$. 
The corresponding \emph{magnetic field} is the 2-form $B=dA\in\Omega^2(G)$. 
Hence $B$ is a smooth mapping $X\mapsto B_X$ from~$\gg$ into the space of 
all skew-symmetric bilinear functionals on $\gg$ 
such that 
$$(\forall X,X_1,X_2\in\gg)\quad 
B_X(X_1,X_2)=\langle A'_X(X_1),X_2\rangle-\langle A'_X(X_2),X_1\rangle.$$
\item We also need the mappings 
$$\theta_0\colon\gg\times\gg^*\to\Fc, \quad 
\theta_0(X,\xi)=\xi+\langle A_\bullet,(R_\bullet)'_0X\rangle$$
and 
$$\theta\colon\gg\times\gg^*\to\mg, \quad(X,\xi)\mapsto(\theta_0(X,\xi),X).$$ 
Here $R_Y\colon\gg\to\gg$, $Z\mapsto Z\ast Y$, is the translation to the right 
defined by any $Y\in\gg$. 
\end{enumerate}
We refer to \cite{BB09a} for more details on these notions.
\qed
\end{setting}

\begin{remark}\label{sympl_fourier}
\normalfont
Let $\Xi:= \gg\times \gg^*$ and denote  the duality between $\gg$ and $\gg^*$ also by 
$$\gg^* \times \gg\to{\RR},\quad (\xi,X)\mapsto \hake{\xi,X}.$$ 
The mapping
$ \hake{\cdot,\cdot}\colon \Xi\times \Xi\to {\RR}, \quad 
\scalar{(X_1, \xi_1)}{ (X_2, \xi_2)}= \scalar{\xi_1}{X_2}- \scalar{\xi_2}{X_1} $ 
defines a symplectic structure on $\Xi$.  
The corresponding Fourier transform is given by 
$$ (F_\Xi a)(X, \xi) =\hat a(X, \xi) = \int\limits_{\Xi} \ee^{-\ie \scalar{(X, \xi)}{ (Y, \eta)}} 
a(Y, \eta) \, \de(Y, \eta), \quad a \in L^1(\Xi).
$$
This transform extends to an invertible operator $\Sc'(\Xi) \to \Sc' (\Xi)$, 
$F_\Xi^{-1}= F_\Xi$ 
 and we denote 
$\check a= F_{\Xi}^{-1}a$. 
\qed
\end{remark}

In the present framework we can make the following definition 
similar to Definition~\ref{calc_def}. 
To emphasize the close relationship between these two constructions, 
let us mention that the representation~$\pi$ introduced in Setting~\ref{o1}\eqref{o1_item6} 
is naturally associated with a certain \emph{finite-dimensional} coadjoint orbit $\Oc$ 
of the infinite-dimensional Lie group $M=\Fc\rtimes_\lambda G$ 
and there exists a canonical symplectomorphism $\Oc\simeq T^*G$. 
(See Prop.~2.9 and subsect~2.4 in \cite{BB09a}.)

\begin{definition}\label{magn_calc}
\normalfont 
For every $a\in\Sc(\gg\times\gg^*)$  
there exists a linear operator $\Op^A(a)$ 
in $L^2(\gg)$ defined by 
\begin{equation}\label{weyl_loc_special}
\Op^A(a)f
=\int\limits_{\Xi} \check{a}(X, \xi)\pi(\exp_M\theta(X, \xi))f\,\de(X, \xi)
\end{equation}
We will call $\Op^A(a)$  
a \emph{magnetic pseudo-differential operator} with respect to 
the magnetic potential~$A$. 
The function $a$ is the \emph{magnetic Weyl symbol} of the pseudo-differential operator $\Op^A(a)$, 
and the \emph{Weyl calculus with respect to the magnetic potential $A$} is 
the mapping $\Op^A$ which takes a function $a\in\Sc(\gg\times\gg^*)$ 
into the corresponding pseudo-differential operator. 
By using a duality reasoning (see \cite{BB09b}), 
one can extend this definition to every tempered distribution $a\in\Sc'(\gg\times\gg^*)$ 
in order to get a continuous linear pseudo-differential operator $\Op^A(a)\colon\Sc(\gg)\to\Sc'(\gg)$. 
\qed
\end{definition}

\begin{remark}\label{magn_pred}
 \normalfont
The role of $\gg_e$ of Definition~\ref{calc_def} is played here by the 
\emph{magnetic predual} 
$$\Oc_*^A:=\{(\bar\theta_0^A(\xi,X),X)\mid X\in\gg,\xi\in\gg^*\}
\subseteq\Fc\rtimes_{\dot\lambda}\gg=\mg$$ 
for the orbit $\Oc$. 
The set $\Oc_*^A$ is just a ``copy'' of $\Oc$ contained in the Lie algebra~$\mg$ of 
the infinite-dimensional Lie group $M$. 
In the general case, if two magnetic potentials give rise to the same magnetic field, 
then the corresponding copies of $\Oc$ in the Lie algebra $\mg$ 
are moved to each other by the adjoint action of the Lie group~$M$ 
(see Rem.~3.5 in \cite{BB09a}). 
This leads to the \emph{gauge covariance} (Theorem~\ref{main1}\eqref{main1_item1_1/2} below) 
of the pseudo-differential calculus 
which can also be described  by the formula 
$$\Op^A(a)f=\int\limits_{\Oc_*}\check{a}(v)\pi(\exp_M v)f\,\de v$$
obtained from \eqref{weyl_loc_special} 
after the change of variables $v=\theta(x,\xi)$.  
\qed
\end{remark}

\begin{theorem}\label{main1}
The Weyl calculus $\Op^A$ has the following properties: 
\begin{enumerate}
\item\label{main1_item0} 
For $P_0\in\gg$ let us denote by $A(Q)P_0$ the multiplication operator defined by 
the function $Y\mapsto\scalar{A_Y}{(R_Y)'_0P_0}$. 
Then the usual functional calculus for the self-adjoint operator 
$-\ie\dot{\lambda}(P_0)+ A(Q)P_0$ in $L^2(\gg)$  
can be recovered from $\Op^A$.
\item\label{main1_item1_1/2} 
Gauge covariance with respect to the magnetic potential $A$: 
If $A_1\in\Omega^1(\gg)$ is another magnetic potential with
$dA=dA_1\in\Omega^2(\gg)$ 
and the function $Y\mapsto\langle A_Y,(R_Y)'_0X\rangle$ belongs to $\Fc$ for every $X\in\gg$, 
then there exists 
$\psi\in\Fc$ such that unitary operator $U\colon L^2(\gg)\to L^2(\gg)$ 
defined by the multiplication by the function $\ee^{\ie\psi}$ satisfies the condition
$U\Op^A(a)U^{-1}=\Op^{A_1}(a)$ for every symbol $a\in\Sc(\gg\times\gg^*)$. 
 \item\label{main1_item1}
 If $\Cpol(\gg)\subseteq\Fc$ and the function $Y\mapsto\langle A_Y,(R_Y)'_0X\rangle$ belongs to $\Cpol(\gg)$ 
for every vector $X\in\gg$, 
then
for every symbol $a\in\Sc(\gg\times\gg^*)$ the magnetic pseudo-differential operator $\Op^A(a)$ 
is  bounded linear on $L^2(\gg)$ and is defined by an integral kernel 
$K_a\in\Sc(\gg\times\gg)$ given by formula 
\begin{equation*}
K_a(X,Y)=\alpha_A(X,Y)\int\limits_{\gg^*}
\ee^{\ie\hake{\xi,X\ast (-Y)}}a\Bigl( \int \limits_0^1 (s(Y\ast (-X))) \ast X \, \de s,\xi \Bigr)\,\de\xi
\end{equation*}
were we have used the notation 
\begin{equation*}
\alpha_A(X,Y) = 
\exp\Bigl({\ie \int\limits_0^1 \scalar{A((s(Y\ast(-X)))\ast X)}{(R_{(s(Y\ast(-X)))\ast X})'_0(X\ast (-Y))} \, \de s}\Bigr)
\end{equation*} 
for every $X,Y\in\gg$. 
\item\label{main1_item2} 
Under the hypothesis of the above Assertion~\ref{main1_item1}., 
the correspondence  $a\mapsto K_a$ is an isomorphism of Fr\'echet spaces 
$\Sc(\gg\times\gg^*)\to\Sc(\gg\times\gg)$
and extends to a unitary operator 
$L^2(\gg\times\gg^*)\to L^2(\gg\times\gg)$. 
\item If the distribution $a\in\Sc'(\gg\times\gg^*)$ takes real values on the real valued functions, 
then the pseudo-differential operator $\Op^A(a)$ is symmetric, 
in the sense that its distribution kernel $K_a\in\Sc'(\gg\times\gg)$ has 
the following symmetry property: 
$$(\forall f,\phi\in\Sc(\gg))\quad 
\hake{K_a,f\otimes\bar{\phi}}=\overline{\hake{K_a,\phi\otimes\bar{f}}}. $$
\end{enumerate}
\end{theorem}

\begin{proof}
See \cite{BB09a} and \cite{BB09b}. 
\end{proof}

\subsection{Magnetic Moyal product}

\begin{definition}\label{diez}
\normalfont
Let us assume that $\Cpol(\gg)\subseteq\Fc$ and $A\in\Omega^1(\gg)$ 
has the property that 
the function $Y\mapsto\langle A_Y,(R_Y)'_0X\rangle$ belongs to $\Cpol(\gg)$ 
for every $X\in\gg$. 
It follows by Theorem~\ref{main1}\eqref{main1_item1} that 
for every $a_1,a_2\in\Sc(\gg\times\gg^*)$ there exists a unique function 
$a_1\#^A  a_2\in\Sc(\gg\times\gg^*)$ such that 
$\Op^A(a_1)\Op^A(a_2)=\Op^A(a_1\#^A  a_2)$ 
and the \emph{magnetic Moyal product}
$$\Sc(\gg\times\gg^*)\times\Sc(\gg\times\gg^*)\to \Sc(\gg\times\gg^*),\quad 
(a_1,a_2)\mapsto a_1\#^A  a_2 $$
is a bilinear continuous mapping. 
\qed
\end{definition}

\begin{theorem}\label{2step_th}
Let $\gg$ be a two-step nilpotent Lie algebra, 
that is, $[\gg,[\gg,\gg]]=\{0\}$. 
If $\Cpol(\gg)\subseteq\Fc$ and $A\in\Omega^1(\gg)$ 
is a magnetic potential with $\langle A(\cdot),X+\frac{1}{2}[X,\cdot]\rangle\in\Cpol(\gg)$ for every $X\in\gg$,  
then the following assertions hold: 
\begin{enumerate}
\item\label{2step_th_item1} 
For every $a\in\Sc(\gg\times\gg^*)$ 
the integral kernel of the bounded linear operator 
$\Op^A(a)\colon L^2(\gg)\to L^2(\gg)$ 
is given by the formula 
\begin{equation*}
K_a(X,Y)=\alpha_A(X,Y)\int\limits_{\gg^*}
\ee^{\ie\hake{\xi,X\ast (-Y)}}a\Bigl(\frac{1}{2}(X+Y),\xi \Bigr)\,\de\xi
\end{equation*}
where for arbitrary $X,Y\in\gg$ we have denoted 
\begin{equation*}
\alpha_A(X,Y) = 
\exp\Bigl(-{\ie \int\limits_0^1 \scalar{A(sY+(1-s)X)}{Y-X} \, \de s}\Bigr).
\end{equation*}
\item\label{2step_th_item2} 
Set 
$$\beta_A\colon\gg\times\gg\times\gg\to\CC,\quad 
\beta_A(X,Y,Z)=\alpha_A^{-1}(X,Y)\alpha_A(Y,Z)\alpha_A(Z,X). $$
If $a,b\in\Sc(\gg\times\gg^*)$ then for every $(X,\xi)\in\gg\times\gg^*$ 
we have 
\begin{equation*}
\begin{aligned}
(a\#^A b)(X,\xi)
=\iiiint\limits_{\gg\times\gg\times\gg^*\times\gg^*}
&a(Z,\zeta)b(T,\tau)\ee^{2 \ie \scalar{(Z-X,\zeta-\xi)}{(T-X,\tau-\xi)}} \\
&\times\ee^{-\ie (\scalar{\xi+\zeta}{[X,Z]}+ \scalar{\zeta+\tau}{[Z,T]}+
\scalar{\tau+\xi}{[T,X]})} \\
&\times \beta_A(Z-T+X, T-Z+X,Z+T-X)
\,\de Z\de T\de\zeta\de\tau.
\end{aligned}
\end{equation*} 
\end{enumerate}
\end{theorem}

\begin{proof}
See \cite{BB09a}. 
\end{proof}


\subsection*{Acknowledgments} 
Partial financial support from the grant PNII - Programme ``Idei'' (code 1194) 
is acknowledged.

\end{document}